\documentclass[12pt]{article}
\usepackage{amssymb,amsthm}
\usepackage{multicol}
\usepackage{graphicx}
\usepackage{amsmath, epsfig}
\usepackage{pgf}
\usepackage{tikz}
\usetikzlibrary{decorations.markings}

\newtheorem{theorem}{Theorem}
\newtheorem{prop}[theorem]{Proposition}
\newtheorem{lemma}[theorem]{Lemma}
\newtheorem{cor}[theorem]{Corollary}

\newcommand{\BE}[1]{\begin{equation} \label{#1}}
\newcommand{\EE}{\end{equation}}
\newcommand{\EB}{\text{EB}}
\newcommand{\TR}{\text{TR}}

\newcommand{\EL}{{\rm E}_{-}}

\newcommand{\E}{{\bf E}\,}

\newcommand{\eps}{\varepsilon}

\newcommand{\LL}{{\rm L}_{-}}

\newcommand{\Prb}{{\bf Pr}}
\newcommand{\R}{{\bf R}}

\newcommand{\sgn}{{\rm sign}}

\newcommand{\Z}{{\bf Z}}
\newcommand{\Xn}{X[n]}

\newcommand{\remark}[1]{\medskip \noindent {\bf Remark #1} }

\title{Explicit error bounds for lattice Edgeworth expansions}
\author{J.P.~Buhler, A.C.~Gamst, R.L.~Graham, and A.W.~Hales}

\begin{document}
\maketitle

\begin{abstract}
Motivated, roughly, by comparing the mean
and median of an IID sum of bounded lattice random variables, we develop
explicit and effective bounds on the errors involved in the one-term
Edgeworth expansion for such sums.
\end{abstract}

% Sawtooth function \reflectbox{N}

Let $X$ be a bounded integer-valued random variable (these will occasionally 
be referred to below as ``dice''), and let
$X[n]$ denote the sum of $n$ independent and identically distributed (IID) 
copies of~$X$ (sometimes referred to as ``rolls'').
If $p_x$ denotes $\Prb(X = x)$, then we say that $x$ is a {\em value} of~$X$ if
$p_x > 0$, and the mean of~$X$ is 
\[
\E X  = \mu_1(X) = \mu_1 = \sum_x \; x \, p_x ,
\]
the higher (central) moments of $X$ are 
$\mu_k(X) = \mu_k = \E (X-\mu_1)^k$, and the standard deviation 
$\sigma = \sigma(X)$ is the square root of $\mu_2$.  (Here, and below, the 
random variable $X$ is omitted from the notation if it can be inferred easily from
context.)

The {\em tilt} of $X$ is
\[
T(X) := \Prb(X > \mu_1) - \Prb(X < \mu_1).
\]
Modulo the annoying question of the precise definition of the median, the sign of $T$ 
measures whether the median is to the right or left of the mean.  
The primary focus of this paper is on the sign of the tilt of $X[n]$ for large~$n$,
which will be denoted
\[
T_n = T_n(X) = T(X[n]). 
\]
As $n$ goes to infinity the Central Limit Theorem says that the distribution
of $X[n]$, suitably shifted and scaled, converges to the standard normal
distribution.  
This implies that the tilt goes to zero as $n$ goes to infinity, and this 
will require us to accurately estimate $T_n$ in order
to say anything at all about its sign. 
This will be done by using the one-term Edgeworth expansion (an asymptotic refinement
of the Central Limit Theorem) for lattice random variables.  The key goal of this
paper is to prove explicit formulas for the error in these approximations sufficient
to enable the determination of the sign of $T_n$, for {\em all}~$n$.

It turns out that the sign of $T_n$ for large~$n$ depends on
the third moment $\mu_3$ (associated with the ``skew'' or ``tilt'' of the distribution)
but also, perhaps more surprisingly, on the congruence class of~$n$ modulo the so-called
``span'' of~$X.$  We will see that for large~$n$ the sign of the tilt is (almost always)
completely determined by these two pieces of data.

To make this more precise, it is convenient to make some definitions.
The {\em span} of a die $X$ is the largest integer $b$ such that all values of $X$ are contained
in a single congruence class modulo~$b$, i.e., an arithmetic progression of the form 
$a+bn$, for~$n \in \Z$.
The integer $a$ is called a {\em shift} of $X$, and it is only well-defined modulo~$b$.
It isn't hard to see that the span is the gcd (greatest common divisor) of all $x-a$ as $x$ ranges
over the values of $X$.  (For the span to be nonzero, $X$ has to have at least two values,
and we will always assume that this is the case.)

If $x$ is an integer and $b$ is any positive integer then let $x \bmod b$ denote the 
unique integer congruent to~$x$ modulo~$b$ that is in the interval $[0,b)$.

\begin{theorem}
% return ((X.b-c)%X.b-c%X.b)/X.sigma - self.mu3/3
\label{thm:TnLim}
Let $X$, $\sigma$, $a$, $b$, and $\mu_3$ be as above.
Then for positive integers $n$,
\[
T_n = \frac{L(na)}{\sqrt{2\pi n}} + E
\]
where
\[
L(c) =   \frac{(-c)\bmod b - c \bmod b}{\sigma}  - \frac{\mu_3}{3\sigma^3},
\]
and
\[
E = o(1/\sqrt{n}).
\]
\end{theorem}

Note that $L(na)$ only depends on the congruence class of~$n$ modulo~$b$,
and that if $n$ goes to infinity in a fixed congruence class 
then the limit of $\sqrt{2\pi n}\,T_n$ exists and is equal to $L(na)$.
The error $E$ will turn out to be bounded by terms that are, roughly, constant
multiples of $1/n$ and $\exp(-c\sqrt{n})/n^p$ for various $c$ and $p$.

Proofs of the lattice Edgeworth expansions in the literature do not seem to 
include explicit error bounds on the error~$E$, and our goal is to exhibit such 
bounds, for bounded lattice random variables.
Such bounds are necessary if one wants to find an $n_0$ together with a
{\em proof} that 
\[
n \ge n_0, \; an \equiv c \bmod b \quad \mbox{imply that} \quad
\sgn(T_n) = \sgn(L(c)).
\]
Briefly, one could say that ``asymptopia'' has arrived when the sign of $T_n$ is equal to
its asymptotic sign.

This question arose for us in \cite{BGH} where the existence of ``maximally
intransitive'' dice was shown.
We\footnote{Well, especially RLG} felt that it should be possible to
determine when asymptopia arrives, i.e., when
the desired dominance relation between the dice constructed in [MID]
was absolutely guaranteed for $n \ge n_0$ (see that article for details).

It is possible for $L(c)$ to be zero, though ``unlikely'' if $X$ is not symmetric.
In this case, higher order Edgeworth expansions are necessary, and this case will be left
to the motivated and energetic reader.

No prior understanding of Edgeworth expansions is required to read this paper, and we
consider only a specific case.  For a broader perspective, the reader could consult
\cite{Feller}, \cite{Petrov},  \cite{Bhattacharya}.
The techniques described here should be applicable more generally.

The first section below develops some of the basic ideas necessary needed to approximate 
$\Prb(X[n] < 0)$.  The second section proves such an approximation, together with
explicit error bounds.  The third section applies this to prove a refined version 
of Theorem~1, and looks at examples.

\section{Preliminaries}

It is convenient to focus on the case of real-valued $X$ with mean~0 and span~1.
If $Y$ is a die with span~$b$ then 
\[
X := \frac{Y -  \mu_1(Y)}{b}
\]
is a bounded lattice random variable with mean~0 and span~1, which says that the values of $X$
lie in a lattice $a+\Z$ but are not contained in a larger lattice $a+d\Z$, $d > 1$.
The tilt is invariant under affine transformations $X \to bX+a$ so
$T_n(Y) = T_n(X)$. 
It is convenient to fix this situation from now on:
$X$ will be a random variable with span~1, mean~0, and shift~$a$.
If $X$ arises from dice as above then the shift $a$ is a rational number.
In this case it might make sense to take a limit as $n$ goes to infinity
through a set of values where $\{na\}$ is fixed.  Here $\{x\}$ denotes the
{\em fractional part} of $x$, i.e., the unique $y$ such that $x = y+j$ for
some integer~$j$, and $0 \le y < 1$.
However, the explicit estimates apply for irrational~$a$ and an arbitrary~$n$, and
may be useful in other contexts.

The central limit theorem says that the {\em cumulative probability 
distribution}  of the normalized random variable $Z_n := X[n]/(\sigma \sqrt{n})$
approaches that of a standard normal random variable in the sense that
\[
\lim_{n \to \infty} \Prb\left(Z_n  \le s \right)  = 
\frac{1}{\sqrt{2\pi }} \int_{-\infty}^s e^{-t^2/2} dt, \qquad \mbox{for all} \; s \in \R.
\]
Our interest in the tilt suggests focusing on the mean $s = \mu_1(Z_n) = 0$.
Then $ \lim_{n \to \infty} \Prb\left(Z_n  \le 0 \right)  = 1/2$.  
The Berry-Esseen Theorem gives an explicit bound on the error, i.e., in the
case $s = 0$,
\[
\left | \Prb\left(Z_n  \le 0 \right)  - 1/2 \right | \le \frac{c}{\sqrt{n}}
\]
for a small constant~$c$, e.g., $c = 3$ in \cite{Feller}.
However, it is easy to show that the tilt $T_n$ is $O(1/\sqrt{n})$, so the Berry-Esseen
level of accuracy is insufficient for saying anything about the tilt
The central result of this paper is of the form
\[
\left | \Prb\left(Z_n  \le 0 \right)  - \left(1/2+\frac{\ell}{\sqrt{n}}\right) \right | \le E(n).
\]
Here $\ell$ depends on on the second, third, and fourth moments of $X$ and 
the fractional part $\{na\}$.
The error $E$ is bounded by an expression whose principal term is of the form $d/n$,
with about 7 other terms that are each of the form
$\lambda  e^{-\tau n^{\gamma}}/n^\rho$
for various constants $\lambda$, $\tau$, $\rho$, and $\gamma$.
Although this is a very special case of the central limit theorem, the techniques
should apply more generally.

As will be seen, this explicit bound on the error in the simplest non-trivial
Edgeworth expansion allows us to prove theorems about the sign of the tilt.

Readers might remember that the skewness of the distribution of 
$\Xn$ depends on the third moment of~$X$.
This is reflected in the term $\mu_3/3\sigma^3$ of $L(c)$ in 
Theorem~\ref{thm:TnLim}.
One intuitive way to see that the third moment and asymptotic tilt  might have
opposite signs
is that for large~$n$ the distribution of $\Xn$ should be approximately
normal, and if the median is slightly negative then the positive values
have to be somewhat larger to make the mean equal to~0, so the
third moment will be positive.

The other term in $L(c)$, which becomes $(\{-na\}-\{na\})/\sigma$ 
in the span~1 case, shows that for lattice random variables
the third moment does not give the full story.
This term is sometimes called the lattice correction term.
To get an intuitive feel for this term, consider
a lattice random variable $X$ of span 1 and shift~$a$  with mean and third moment
equal to~0 (a simple linear algebra exercise shows that these are easy to come
by).  
Since $\mu_3 = 0$ the lattice correction term is the only term.
The support of the probability distribution of $X[n]$ is contained in the set of
real numbers $x$ that are congruent to $na$ modulo~$1$.  For large $n$, the
probability distribution is close to a re-scaled version of the standard normal
distribution.  To first order, it seems reasonable to suspect that if 
$c = {na} $ is less than $1/2$ then the sum of the probabilities $p_x = \Prb(X= x)$
for $\{x\} = c$, $x > 0$, will be slightly greater than the corresponding
sum for $x < 0$ since the latter contribute more heavily to making the mean~0,
which tends to make the tilt positive.

Of course the explicit formula for $L(c)$ emerges cleanly from the calculations below, 
and perhaps this supersedes all of these heuristic remarks!

\subsection{The characteristic function}

As above, fix $X$ with mean 0, span 1, shift $a$, central moments $\mu_k$, and standard
deviation~$\sigma$ defined by $\sigma^2 = \mu_2$.
The goal of this section is to give a formula for $\Prb(X[n]) < 0$ in terms of an
integral.
This formula can be proved fairly directly using Fourier series, but we will take
a somewhat more leisurely approach that starts with a contour integral.

The probability generating function (PGF) of~$X$ is a function of a complex variable~$z$ 
defined by 
\[
F(z) = \E z^X = \sum_x p_x z^x = z^a \, \sum_j p_{a+j} z^j,
\]
using the fact that values of~$X$ can be written $x = a+j$ for some integer~$j$.
Note that $z^{-a} F(z)$ is a finite Laurent series:
\[
z^{-a}F(z) = \sum_x p_x z^{x-a} = \sum_j p_{a+j} z^j.
\]
Applying Cauchy's Theorem gives
\[
p_{a+j} = \Prb(X = a+j) = [z^j] \, z^{-a} F(z) = \frac{1}{2\pi i} \, 
    \oint_{\gamma} \, \frac{z^{-a} F(z)}{z^{j}} \, \frac{dz}{z} 
\]
where $[z^j] z^{-a} F(z)$ denotes the coefficient of $z^j$ in the polynomial $z^{-a}F(z)$, and the contour 
$\gamma$ can be chosen to be a counterclockwise circle around the origin.

With an eye to ultimately applying this to the tilt,
we use this integral to find a useful expression for $\Prb(X < 0)$.
The set of negative values $x = a+j$ is the set of $\{a\}+j$ as $j$ ranges over negative
integers (where $\{a\}$ is the fractional part of~$a$, defined above; this might reasonably
be denoted $a \bmod 1$).
Therefore,
\begin{align*}
\Prb(X < 0) &  = \sum_{j<0} p_{\{a\}+j} =  \frac{1}{2\pi i} \, \oint \,  \left( z + z^2 + z^3 \ldots \right )
  \, z^{-\{a\}} \, F(z) \, \, \frac{dz}{z} \\
   & = \frac{1}{2\pi i} \, \oint \, \frac{z^{1-\{a\}}\,F(z)}{1-z} \, \frac{dz}{z} 
\end{align*}
where the radius is less than~$1$ to ensure that the geometric series converges.

If $n$ is a positive integer and $X[n] := \sum X_i$, where the $X_i$ are $n$ independent 
random variables, then independence implies that the PGF of $X[n]$ is $F(z)^n$.
Applying the preceding formula to $\Xn$ gives
\[
\Prb(\Xn < 0) = 
\frac{1}{2\pi i} \, \oint \, \frac{z^{1-\{na\}}\,F(z)^n}{1-z} \, \frac{dz}{z} 
\]
where the contour is a counterclockwise circle around the origin of radius slightly less than~1.
%\begin{center}
%\includegraphics[width=1.0in]{fig1.pdf}
%\end{center}
 \begin{center}
 \begin{tikzpicture}
   \draw (0,-1.5) -- (0,1.5);
   \draw (-1.5,0) -- (1.5,0);
   \begin{scope}
        \clip (-1,-1) rectangle (.982,1);
        \draw[decoration={markings,mark=at position 0.5 with {\arrow{<}}}, postaction={decorate}](1,0)[thick] circle(0.2);
        \draw[decoration={markings,mark=at position 0.375 with {\arrow{>}}}, postaction={decorate}](0,0)[thick] circle(1);
   \end{scope}
   %\draw(1.03,0)[white, thick,fill=white] circle(0.2);
 \end{tikzpicture}
 \end{center}
Move the contour outward to the unit circle except for an infinitesimal
semicircular divot centered at, and to the left of,~1.  
In other words, the contour follows the unit circle counterclockwise from
$z = e^{i\eps }$ to $z = e^{-i\eps }$ followed by a clockwise small circular arc back to
$e^{i\eps}$.  For very small $\eps$ the integrand is close to
$-1/(z-1)$, and the contour is basically a small clockwise semicircle; Cauchy's 
Theorem implies that the value of the integral over the divot is very close to $1/2$.
Taking the limit as $\eps$ goes to zero gives
\[
\Prb(\Xn < 0)  = \frac{1}{2} + \frac{1}{2\pi i} \, \oint \, \frac{z^{1-\{na\}}F(z)^n}{1-z} \, \frac{dz}{z} 
\]
where the contour is the unit circle punctured at $z = 1$, with the 
``principal value interpretation'' at the puncture.
With an eye to changing variables by $z = e^{it}$, let
\[
f(t) = F(e^{it}) = \E e^{itX} = \sum_x p_x \, e^{itx}
\]
be the {\em characteristic function} (CF) of~$X$.
Then
\[
\Prb(\Xn < 0)  
= \frac{1}{2} - \frac{1}{2\pi i } \, \int_{-\pi}^\pi \, e^{i \alpha t} \, f(t)^n\, 
D(t) \, \frac{dt}{t}
\]
where $\alpha = \alpha_n = 1/2-\{na\}$ and 
\[
D(t) = (t/2)/\sin(t/2).
\]  
The principal value interpretation of the integral at $t = 0$ will always be used, 
which means that 
\[
\int_{-\pi}^\pi := \lim_{\eps \to 0} \left ( \int_{-\pi}^{-\eps} + \int_{\eps}^{\pi} \right).
\]
The following result summarizes the above discussion.

\begin{theorem}
\label{thm:TnCnt}
With the above notation,
\begin{equation}
\label{eq:I1r}
\Prb(\Xn < 0)  = 1/2-I_0, \quad \mbox{where} \quad 
I_0 := \frac{1}{2\pi i } \, \int_{-\pi}^\pi \, e^{i \alpha t} \, f(t)^n\, D(t) \, \frac{dt}{t} .
\end{equation}
In addition, if $\alpha' = 1/2-\{-na\}$, then
\[
T_n  = \frac{1}{2\pi i } \, \int_{-\pi}^\pi \, e^{i \alpha t} \, f(t)^n\, D(t) \, \frac{dt}{t}  -
\frac{1}{2\pi i } \, \int_{-\pi}^\pi \, e^{i \alpha' t} \, f(-t)^n\, D(t) \, \frac{dt}{t}  .
\]
\end{theorem}

The last statement of the theorem follows easily from the first, using several obvious facts:
(1) $\Prb(\Xn > 0 ) = \Prb((-X)[n] < 0)$, (2) $\alpha'$ is the shift of $-X$, and 
(3) the CF of~$-X$ is $f(-t)$.

For later use, we remark that
$D(t)$ is even, has $D(0)=1$, 
$D(\pi) = \pi/2$,  and
has power series coefficients that can be expressed in terms of Bernoulli numbers and
are positive.  From that, or alternately by a simple calculus exercise, it follows
$D(t)$ is increasing on $[0,\pi]$ so that
\BE{ineq:D}
D(t)\le D(\pi) = \pi/2 \qquad \mbox{on} \qquad [-\pi,\pi].
\EE

\subsection{Span}

As above, $X$ has span 1, mean 0, shift~$a$, and CF $f(t)$.

\begin{lemma}
\label{lem:span}
There are integers $c_x$, one for each value $x$ of  $X$, such that
\[
\sum_x c_x = 0, \quad \mbox{and} \quad \sum_x c_x x = 1.
\]
\end{lemma}

\begin{proof}
Let $y$ be a value of~$X$.
If $b := \gcd(\{x-y\})$ is larger than 1 then the values of $X$ are contained in $y+b\Z$
which contradicts the fact that $X$ has span~1.
Therefore, the gcd is~$1$ and there are integers $c_x$, for $x \ne y$, such that
\[
\sum c_x (x-y) = 1.
\]
Set $c_y = -\sum_{x \ne y} c_x$.
The stated properties are easily verified.
\end{proof}

A set $\{c_x\}$ as in the lemma is said to be a {\em certificate} of the fact that $X$ has span~1.

\begin{lemma} 
The function $|f(t)|$ has period~$2\pi$, and $|f(t)|<1$ for $t \in (0,2\pi)$.
\end{lemma}

\begin{proof}
We can assume that the shift $a$ is a value of~$X$.
Since
\[
f(t) = e^{iat} \sum_x p_x e^{it(x-a)}
\]
and the $x-a$ are all integers it follows that $|f(t+2\pi)| = |f(t)|$, and that 
the period of $|f(t)|$ is of the form $2\pi/b$ for some positive integer~$b$.  
Then
\[
1 = | f(2\pi/b) | = \left | e^{i2\pi a/b} \sum_x p_x e^{ 2\pi i (x-a)/b} \right | \le
\sum p_x = 1.
\]
Equality in this use of the triangle inequality implies that all $e^{\pi i(x-a)/b}$ are equal 
to 1, i.e., that all $x-a$ are multiples of~$b$.  If $b>1$ then this contradicts the fact
that the span of~$X$ is equal to~1.
\end{proof}

These lemmas show that if the span is~~$1$ then
$\gcd(\{x-x'\}) = 1$, $X$ has a certificate, and $|f(t)|$ has period $2\pi$.
It is not hard to show that any of these implies that the span is~$1$,
so all four conditions are equivalent.

\subsection{Bounding CF power series tails}

The power series of the CF for $X$ converges for all real~$t$, and has the form
\BE{eq:CFps}
f(t) = \sum_x \, p_x \, e^{itx} =  
\sum_{k\ge 0} \mu_k \, \frac{(it)^k}{k!} = 1 - \frac{\mu_2 t^2}{2} 
- i \, \frac{\mu_3 t^3}{6} + \frac{\mu_4 t^4}{24} + \ldots.
\EE
The tail of this power series has an especially tight bound, saying that
the remainder after $k$ terms is at most the absolute value of the next 
term with the moment replaced by the corresponding absolute moment.

\begin{lemma}
\label{lem:CkB}
Let $\overline\mu_k = \E |X|^k$ denote the $k^{\rm th}$ absolute moment of~$X$.
Then
\[
\left|\sum_{j \ge k} \frac{\mu_j (it)^j}{j!} \right | \le
\frac{\overline\mu_k |t|^k}{k!} .
\]
\end{lemma}

\begin{proof}
The expansion
\[
e^{it} = \sum_{0 \le j < k} \, \frac{(it)^j} {j!} + \theta(t) \,  \frac{(it)^k}{k!}, \qquad |\theta(t)| \le 1,
\]
follows from several standard integral forms of the remainder in Taylor's Theorem.
Replace $t$ by $tx$, where $x$ is a value of $X$, multiply by $p_x$, and sum over
all values $x$ to arrive at
\[
f(tx)  = \sum_{j < k} \sum_x p_x \frac{(itx)^j}{j!} + \sum_{x} \theta(tx) p_x \frac{(itx)^k}{k!}  
= \sum_{j < k} \mu_j  \frac{(it)^j}{j!} + \sum_{x} \theta(tx) p_x \frac{(itx)^k}{k!}  .
\]
Taking the absolute value of the remainder gives the bound
\[
\left | \sum_{x} \theta(tx) p_x \frac{(itx)^k}{k!}  \right | \le
\sum_{x}  p_x |x|^k \frac{|t|^k}{k!}  = \overline \mu_k |t|^k/k!
\]
as claimed.
\end{proof}

\subsection{Bounding the CF}

Fix a certificate $\{c_x\}$, as above, and let
$C = \sum_x |c_x|$ be its $\ell_1$ norm.  Note that $C \ge 2$ since at
least two of the integers $c_x$ are nonzero.
Before proving a bound on the CF $f(t)$ outside a neighborhood of~0,
a preliminary lemma is needed.

\begin{lemma} If $0 < t < \pi$ then 
no interval on the circle of arc length less than $2t/C$ contains $e^{itx}$
for all values~$x$ of~$X$.
\end{lemma}

\begin{proof}
Suppose, to the contrary, that all $e^{itx}$ lie in the interior of the arc
from $e^{iu}$ to $e^{i(u+2t/C)}$ on the circle.  (Since $2t/C < \pi$
the interior is well-defined --- it is the smaller of the arcs into which
those two points divide the circle.)
Then there are integers $j_x$ such that
\[
u < tx + 2\pi j_x < u + 2t/C \; .
\]
Subtract $u+t/C$ to get
\[
-t/C < tx + 2 \pi j_x - u -t/C< t/C.
\]
Multiply by $c_x$ and sum, noting that the inequalities reverse if $c_x$ is
negative, to get
\[
-t = \frac{- \sum |c_x| t}{C}  <  t + 2\pi M < \frac{\sum |c_x| t}{C} =  t
\]
for some integer~$M$.
If $M$ is nonnegative the right inequality is false, and if $M$ is negative
the left inequality is false.  This finishes the proof of the lemma.
\end{proof}

\begin{theorem}
\label{thm:CFbound}
Let $X$ be a random variable as above, and $f(t)$ its CF. 
Let $m:= \min\{p_x\}$ be the minimum probability of a value.
If $|t| \le \pi$ then
\[
\left | f(t) \right | \le 1 - \frac{8m t^2}{\pi^2 C^2}.
\]
\end{theorem}

\begin{proof}
By the preceding Lemma there are two values $x,y$ such that $x< y$ and the ``arc-length''
distance between $e^{itx}$ and $e^{ity}$ on the unit circle is at least $2t/C$ and at most ~$\pi$, i.e.,
\[
\frac{2t}{C} \le ty-tx \le \pi.
\]
Then
\[
f(t) =  \sum_u p_u e^{itu} = T + m e^{itx} + m e^{ity}, 
\]
where $T$ is a trigonometric sum with nonnegative coefficients that sum to $1-2m$, and
therefore
\begin{align*}
|f(t)| & \le 1-2m   + \left | m e^{it(x+y)/2} \left ( e^{it(x-y)/2}+e^{it(y-x)/2} \right ) \right | \\
  & = 1-2m + 2m \cos( t(y-x)/2) \\
  & \le 1-2m + 2m \cos(t/C) = 1 - 2m(1- \cos(t/C)) \\
  & = 1 -4m\sin^2(t/2C).
\end{align*}
Since $x/\sin(x)$ is increasing on $[0,\pi/4]$ it follows that
\[
\sin(t/2C) \ge \frac{t}{2C} \frac{\sin(\pi/4)}{\pi/4} = \frac{\sqrt{2}t}{\pi C}.
\]
Thus
\[
1 -4m\sin^2(t/2C) \le \frac{8mt^2}{\pi^2 C^2},
\]
finishing the proof.
\end{proof}

A similar bound $|f(t)| \le 1 -  d t^2$ can be found in \cite{Benedicks}, for
a completely different constant $d$ (not consistently better or worse than the
constant in the above theorem). 
In addition a technique is given in \cite{Benedicks} to improve the bound when
there are several independent certificates, and that idea also applies to our
bound.
Note that from (\ref{eq:CFps}) it is clear that any such constant has to
be strictly smaller than~$\mu_2/2$.
Later we will see how to, for practical purposes, find bounds that, roughly,
say that in a practical context the constant can be made as close to~$\mu_2/2$ as 
desired.

\subsection{Facts about the Gamma function}

Several facts about values of the Gamma function (and its upper and lower variants) at
integers and half-integers will be needed below.
To slightly complicate matters, these will arise here as integrals of functions of the form
of $t^{x} \exp(-ct^2/2)$.  It is convenient to collect these in one place for the sake of
future reference.
In the Proposition below
the ``double factorial'' $x!!$ of a nonnegative integer denotes the 
product of all positive integers up to~$x$ that have the same parity as~$x$, i.e.,
\[
x!! = \prod_{0 \le k < x/2} \, (x-2k).
\]

\begin{prop}
\label{prop:G}
Let $c$ and $s$ be positive real numbers, and
$k$ a positive integer.
\[
\begin{array}{lll}
(1) & \quad
\int_0^\infty t^{2k} \, e^{-ct^2/2} \, \frac{dt}{t} = c^{-k} \, (2k-2)!!, &
\int_0^\infty t^{2k-1} \, e^{-ct^2/2} \, \frac{dt}{t} = c^{-k+1} \, (2k-3)!! 
{\displaystyle \sqrt{\frac{\pi}{2c}}} \\
(2) & \quad
\int_0^s t^{2} \, e^{-ct^2/2} \, \frac{dt}{t} = {\displaystyle \frac{1}{c} \, \left(1 - e^{-cs^2/2}\right) } &
\int_0^s t^{4} \, e^{-ct^2/2} \, \frac{dt}{t} = {\displaystyle \frac{2}{c^2} } \,
\left ( 1 - e^{-cs^2/2}(1+{\displaystyle \frac{cs^2}{2}}) \right ) \\
(3) & \quad \int_s^\infty  e^{-ct^2/2} \, \frac{dt}{t} \le 
{\displaystyle \frac{e^{-cs^2/2}}{cs^2}}, &
\int_s^\infty  \,t \, e^{-ct^2/2} \, dt = {\displaystyle \frac{e^{-cs^2/2}}{cs}} \\
& \multicolumn{2}{c}{
 \int_s^\infty  t^2 e^{-ct^2/2} \, dt < 
{\displaystyle \frac{e^{-cs^2/2}}{c^2s}\,(1+cs^2).}}  \\
\end{array}
\]
\end{prop}
From now on, WGFF$x$ (``well-known gamma function fact $x$'') will
refer to some statement in part ($x$) of this Proposition.
The only nontrivial part in the entire Proposition is the case $k = 1$ of the second part of (1), 
which is a famous integral.  Everything else follows from elementary integration or a sufficiently
cunning application of the integration by parts formula
\[
\int \, t^{x} \, e^{-ct^2/2}\, \frac{dt}{t} = \frac{-t^{x-2}\, e^{-ct^2/2}}{c} + \frac{x-2}{c} \,
\int t^{x-2} \, e^{-ct^2/2} \, \frac{dt}{t}. 
\]
For instance, the first WGFF3 follows from
\BE{wgff3proof}
\int_s^\infty  e^{-ct^2/2} \, \frac{dt}{t} 
= \left [ \frac{- e^{-ct^2/2}}{ct^2} \right]^\infty_s - 
\frac{2}{c} \int_s^\infty \frac{e^{-ct^2/2}}{t^3} \, dt .
\EE

\section{Main Theorem}

The main technical theorem of this paper, Theorem~\ref{thm:main}, says that 
\[
\Prb(\Xn < 0 ) = \frac{1}{2} - \frac{\LL}{\sqrt{2\pi n}} + \EL(n), \qquad \LL = 
\LL(\{na\}) = \frac{1/2-\{na\}}{\sigma} - \frac{\nu_3}{6}
\]
where explicit bounds on $\EL(n)$ are given.  
As alluded to earlier, the formula for $\LL$ is ``well-known'' from Edgeworth expansions;
the point of the theorem is of course the bound on $\EL$. 
The precise statement of this will be deferred until the end of this section.
This can be used to give an analogous statement for the tilt;
our aim is for this to be good enough to use in practice to find an 
$n_0$ such that the sign of $T_n$ is constant for $n \ge n_0$.

The following subsections (a) introduce a ``scale-invariant'' version of the 
earlier notation and results, (b) outline the steps of the proof of the theorem, 
(c) methodically work through those steps, and then (d) finally give a full statement 
of the theorem.

\subsection{Scale-Invariance}

It is convenient to modify the notation slightly
by introducing {\em scale-invariant} quantities where possible.
Note that $T_n$ is scale-invariant: it is unchanged if $X$ is replaced by a multiple of~$X$.
We introduce scale-invariant versions of
$\mu_k$, $\alpha$, $f(t)$, and~$D(t)$  
by advancing alphabetically:
\[
\begin{array}{l@{\qquad}l}
 \nu_k  := \mu_k/\sigma^k & g(t) := f(t/\sigma) = \sum \nu_k(it)^k/k! \\
 E(t)  := D(t/\sigma) = (t/(2\sigma))/ \sin(t/2\sigma) & \beta := \alpha/\sigma = (1/2-\{na\})/\sigma\\
\end{array}
\]
\noindent The key results of the preceding section using this notation are:
\begin{itemize}
\item The value of the cumulative distribution function of $X[n]$ at 0, in terms of an integral,
Theorem~\ref{thm:TnCnt}, becomes 
\BE{Prb:integral}
\Prb(\Xn < 0)  = 1/2- \frac{1}{2\pi i } \, 
\int_{-\pi\sigma}^{\pi\sigma} \, e^{i \beta t} \, g(t)^n\, E(t) \, \frac{dt}{t} .
\EE
\item
The bound on the tail of the power series of the CF, Lemma~\ref{lem:CkB}, becomes
\BE{tailgbnd}
\left|\sum_{j \ge k} \frac{\nu_j (it)^j}{j!} \right | \le
\frac{\overline\nu_k t^k}{k!} .
\EE
\item
Finally, the bound on the CF, Theorem~\ref{thm:CFbound}, becomes (introducing a factor of 
2 with an eye to the earlier gamma function facts)
\BE{thm:CFgbnd}
|g(t)| \le 1-\frac{r t^2}{2}, \quad \mbox{where} \; r = \frac{16m}{(\pi C \sigma)^2}, \qquad 
\mbox{for} \; |t| \le \pi \sigma.
\EE
\end{itemize}
Since the power series for $g(t)$ starts out with $1 - t^2/2$, the constant $r$
measures how ``tractable'' $X$ is; if $r$ is small then the tail integral estimate
will be weak, and the closer $r$ is to~$1$ the better the estimate will be.

\subsection{Proof Outline}

Throughout the proof, $X$ is a bounded lattice random variable with 
mean~0, span~1, shift~$a$, and scale-invariant CF function $g(t) = \E e^{itX/\sigma}$.   
The standard deviation is $\sigma = \sqrt{\mu_2}$, 
and scale-invariant central moments are $\nu_k = \mu_k/\sigma^k$.
Moreover, when $n$ is given,
$\beta_n = \beta  = (1/2 - \{na\})/\sigma$.
% As earlier, $C$ is the $L_1$ norm of a certificate $\{c_x\}$ for $X$ (i.e., a
% guarantor that~$X$ has span~1), and $m = \min(p_x)$ is
% the minimum probability of the values of~$X$ that are in the support of the certificate.

Fix a positive integer~$n$ and let $s$ be a positive real number~$s$.
During the proof various upper bounds will be placed on~$s$, and it will be assumed
throughout that they hold.  The
bounds on the various approximation errors are in practice smallest
when $s$ is as large as possible.  As will be seen, $s$ will in practice be
a constant (that depends on~$X$) times $n^{-1/4}$.

For the sake of a (reasonably) simple statement of error bounds, no attempt will be
made to optimize various constants that arise.  This seems appropriate in the 
motivating example of determining the sign of the tilt by the fact that nowadays a computer 
can calculate $T_n$ for large $n$.  Thus the point of the estimates is to enable a proof 
of when the asymptotic tilt has arrived so that, when combined with computation, the sign is 
known for all~$T_n$.  Thus finding the absolute best possible 
$n_0$ may not be important.

The technique for proving Theorem~\ref{thm:main} is as follows.
From (\ref{Prb:integral}) above we know that $\Prb(X[n]<0) = 1/2-I_0$, where
\[
I_0  :=  \frac{1}{2\pi i } \, \int_{-\pi\sigma}^{\pi\sigma} \, e^{i \beta t} \, g(t)^n\, 
E(t) \, \frac{dt}{t} .
\]
This integral will be approximated by defining a sequence of further integrals
$I_k$, $1 \le k \le 5$.
Each $I_{k} $ will be a reasonable approximation to $I_{k-1}$, and 
the error between them will be explicitly bounded.  The last integral $I_5$ can
be evaluated directly, and is equal to $\LL/\sqrt{2\pi n}$.
The difference between $I_0$ and $I_5$ is of course bounded by the sum of bounds
on the differences between consecutive integrals, leading to a bound on $\EL$.

The core idea of these approximations is that the dominant contribution to $I_0$
should come from a small interval around~$0$ whose size depends on~$n$.

The {\em modus operandi} of the proof here is then summarized by the following 
sequence of approximations; the subscript on each approximation gives the number 
of the subsection in which the corresponding error is bounded
\begin{alignat*}{3}
I_0 & =  \frac{1}{2\pi i } \, \int_{-\pi\sigma}^{\pi\sigma} \, e^{i \beta t} \, g(t)^n\, 
E(t) \, \frac{dt}{t}   &\;& \\
& \simeq_3 \frac{1}{2\pi i } \, \int_{-s}^{s} \, e^{i \beta t} \, g(t)^n\, 
E(t) \, \frac{dt}{t} &\;& \mbox{tail integral bound}    \\
& \simeq_4 \frac{1}{2\pi i } \, \int_{-s}^{s} \, e^{iu} \,
e^{-n t^2/2} \, E(t) \, \frac{dt}{t},   &\;&
u = \beta t - \frac{n \nu_3 t^3}{6} \\
& = \frac{1}{2\pi } \, \int_{-s}^{s} \, \sin(u) \,
e^{-n t^2/2} \, E(t) \, \frac{dt}{t},   &\;&  \frac{\cos(u)}{t} \; \mbox{is odd} \\
& \simeq_5 \frac{1}{2\pi } \, \int_{-s}^{s} \, \sin(u) \,
e^{-n t^2/2} \, \frac{dt}{t},   &\;& E(t) \simeq 1 \\
& \simeq_6 \frac{1}{2\pi } \, \int_{-s}^{s} \, u \,
e^{-n t^2/2} \, \frac{dt}{t},   &\;& \sin(u) \simeq u \\
& \simeq_7 \frac{1}{2\pi } \, \int_{-\infty}^{\infty} \, (\beta  - n \nu_3 t^2/6) \, 
e^{-n t^2/2} \, dt,   &\;& \mbox{another tail bound.}\\
\end{alignat*}

\subsection{Bounding the tail integral}

For large~$n$ the dominant contribution to the integral
\[
I_0 = \frac{1}{2\pi i } \, \int_{-\pi\sigma}^{\pi \sigma} \, e^{i \beta t} \, g(t)^n\, 
E(t) \, \frac{dt}{t}  ,
\]
should come from a small neighborhood of the origin.  In fact the approximation
\[
g(t)^n \simeq (1-t^2/2)^n \simeq e^{-nt^2/2}
\]
suggests that the width of the neighborhood might be on the order of $1/\sqrt{n}$.
Let $s \le \pi \sigma$ be an arbitrary positive real number and define
\[
I_1 = \frac{1}{2\pi i } \, \int_{-s}^{s} \, e^{i \beta t} \, g(t)^n\, 
E(t) \, \frac{dt}{t}  .
\]
The parameter $s$ will be chosen later, and its optimal value will actually turn out to be
$O(n^{-1/4}).$

Recall from (\ref{thm:CFgbnd}) that $|g(t)| \le 1-rt^2/2$ for $t$ on the interval of 
integration, where the constant $r $ was defined above to be $r = 16m/(\pi C \sigma)^2$,
$m$ is the smallest probability of a point on the support of some certificate, and $C$ is
the $L_1$-norm of that certificate.

\begin{theorem}
\label{thm:I01B}
\[
\left | I_0-I_1 \right | \le \frac{\exp(-nrs^2/2)}{2nrs^2}.
\]
\end{theorem}

\begin{proof}
Use $|g(t)| \le 1-rt^2/2$ and $1-x \le e^{-x}$ to get
\[
|g(t)|^n \le (1-rt^2/2)^n \le \exp(-nrt^2/2).
\]
The difference $I_0 - I_1$ is the sum of a right tail and a left tail integral that
are bounded in exactly the same way.  So it suffices to multiply the bound on the upper tail by 2,
Recall the upper bound $|E(t)| \le \pi/2$, and use the bound on $g(t)^n$:
\begin{align*}
\left| I_1-I_0\right|  & \le  2 \left |\, \frac{1}{2\pi i} \, \int_{s}^{\pi\sigma} \, 
e^{i \beta t} \, g(t)^n\, E(t) \, \frac{dt}{t} \right | \\
& \le 2 \cdot \frac{1}{2\pi} \cdot \frac{\pi}{2}  \; \int_{s}^\infty \, e^{-nrt^2/2} \, \frac{dt}{t} 
\le \frac{e^{-nrs^2/2}}{2 nrs ^2}
\end{align*}
where the last inequality is WGFF3.
\end{proof}

\subsection{Bounding the power series tail}

In order to analyze the $g(t)^n$ term in the integrand it is convenient to introduce
notational shorthand for terms and tails of the power series of $g(t)$:
\[
g_j := \frac{\nu_j(it)^j}{j!}, \quad G_k := \sum_{j \ge k} g_j.
\]
Although this lighter notation is pleasant it is important to remember
that $g_k$ and $G_k$ depend on~$t$.
Note that (\ref{tailgbnd}) above says that $|G_k| \le |g_k| $ for even~$k$, and
$|G_k| \le \overline\nu_k |t|^k/k!$ for odd~$k$.

The critical term in the integrand of $I_1$ is $g(t)^n$, and the purpose of
this section is to bound the error incurred in the approximations
\[
g(t)^n = \exp(n \log(1+G_2)) \simeq \exp(nG_2) \simeq \exp(n(g_2+g_3)).
\]
It is convenient to introduce further notation.  Let
\[
q_1 := \frac{1}{5} + \frac{\nu_4}{24}.
\]
Motivated by replacing $g(t) = \exp(\log(1+G_2))$ by $\exp(g_2 + g_3)$ ,
define a ``remainder''~$R$ by
\[
R = \log(1+G_2)-g_2-g_3.
\]
This should be small if $n$ is large.

The following obvious bound (OB) on tails of power series with positive
coefficients will be used three times below;
the proof is embedded in the statement of the lemma (!).

\begin{lemma} (OB)
Let $P(x) = p_k x^k + p_{k+1} x^{k+1} + \ldots$ be a power series
power with nonnegative coefficients $p_j$, and suppose that $P(y)$ converges for 
some positive real number~$y$.  Then if $|x| \le y$,
\BE{OB}
|P(x)-p_kx^k| = |x|^{k+1} \cdot \left | \frac{P(x)-p_kx^k}{x^{k+1}} \right| 
\le |x|^{k+1} \left | \frac{P(y)-p_k y^k}{y^{k+1}}\right|.
\EE \end{lemma}

\begin{theorem} 
Let $s$ be a positive real number and assume throughout
that $|t|\le s$.  If $s \le 1$ then $|G_2| \le 1/2$, and the
power series for $\log(g(t)) = \log(1+G_2)$ converges.
Moreover, 
\[
|R| \le q_1 t^4, \quad \mbox{where } \;  q_1 := \frac{1}{5}+\frac{\nu_4}{24}, \quad R = \log(1+G_2)-g_2-g_3.
\]
If also $s \le (q_1 n)^{-1/4}$, i.e., $nq_1s^4 \le 1$, then 
\[
\left | e^{nR}-1 \right | \le n p_0 |R| \le n p_0 q_1 t^4, \qquad \mbox{where } \;
p_0 := e-1 \simeq 1.71828.
\]
\end{theorem}

\begin{proof}
Part 1 follows from $|G_2| \le g_2 = t^2/2 \le s^2/2 \le 1/2$ and the fact that
the logarithm series 
\[
\log(1+G_2) = G_2 - G_2^2/2 + G_2^3/3 - \ldots
\]
converges by comparison with a geometric series of ratio $1/2$.

For the second part, first note that
\[
|R| \le |\log(1+G_2)-G_2| + | G_2-g_2-g_3|.
\]
The second term is just $|G_4| \le g_4 = \nu_4 t^4/24$.  The first term can be bounded
by applying the OB to $P(x) = -\log(1-x) = \sum x^k/k$ 
with $k = 1$, $x = -G_2$ and $y = 1/2$.  Note that $|x| = |G_2| \le t^2/2 \le s^2/2 \le 1/2$.
The OB gives 
\begin{align*}
|-\log(1+G_2)+G_2| \le G_2^2 \left( \frac{-\log(1-1/2)-1/2}{(1/2)^2} \right ) 
= \frac{t^4}{4}\left(\frac{\log(2)-1/2}{1/4} \right ) 
 \le \frac{t^4}{5}.
\end{align*}
(By choosing an even smaller bound on~$s$, this could be made as close to $t^4/8$ as desired,
but no lower.)
All in all this gives
\[
|\log(1+G_2)-g_2-g_3| \le q_1 t^4, \quad q_1 = \frac{1}{5} + \frac{\nu_4}{24}
\]
as desired.

For the third part, note that the assumed bound on $s$ implies that
\[
| nR | \le n q_1 t^4 \le n q_1 s^4 \le 1.
\]
Apply OB to $P(x) = e^x$ with $k = 0$, $x = nR$ and $y = 1$ to get
\[
\left| e^{nR}-1 \right | = |nR| \; \left(\frac{e^1-1}{1}\right) \le np_0q_1 t^4, \quad p_0 := e-1
\]
as desired.
\end{proof}

Before using this lemma for the central goal of this section --- bounding the difference between
$I_1$ and a soon-to-be defined $I_2$ --- we prove a corollary that will be used later to improve
the tail integral bound in the previous subsection.

\begin{cor}
\label{cor:tail}
With the above notation, if $s \le 1$ then
\[
\int_s^1 | g(t)|^n \frac{dt}{t} \le e^{-ns^2/2} \; \left ( \frac{1}{ns^2} + \frac{2 p_0 q_1}{n} + p_0q_1s^2 \right ).
\]
\end{cor}

\begin{proof}
The last part of the theorem says that
\[
|g(t)|^n = \left| e^{ng_2+ng_3} \left (1 + ( e^{nR}-1) \right ) \right| 
\le e^{-nt^2/2} \left ( 1 + \theta n p_0 q_1 t^4 \right )
\]
so that
\[
\int_s^1 | g(t)|^n \frac{dt}{t} \le 
\int_s^1 e^{-nt^2/2}  \frac{dt}{t} +
np_0q_1 \; \int_s^1  t^4 e^{-nt^2/2} \frac{dt}{t}.
\]
The first of the integrals on the right hand side can be estimated by (extending the interval to infinity and) using the
first WGFF3, and the second can be evaluated exactly by taking the difference of two instances of the second WGFF2.
The result is
\[
\int_s^1 | g(t)|^n \frac{dt}{t} \le  
\left ( \frac{e^{-ns^2/2}}{ns^2} + \frac{ 2 n p_0 q_1 \; e^{-ns^2/2}}{n^2} \left ( 1 + \frac{ns^2}{2} \right) \right)
\]
which simplifies to the expression in the corollary.
\end{proof}

Returning to the problem of approximating $I_1$, note that
the first two factors in the integrand can be written
\begin{align*}
e^{i\beta t} g(t)^n & = \exp(i\beta t + n \log(1+G_2) ) \\
    & = \exp(i \beta t + ng_2 + ng_3 + nR) \\
    &  = \exp(i u) \exp(-n t^2/2) \exp(nR)
\end{align*}
where the quantity
\[
u = \beta t - n \nu_3 t^3/6
\]
captures the key imaginary terms.
This motivates the definition of the next integral:
\[
I_2 := \frac{1}{2\pi i } \, \int_{-s}^{s} \, e^{i u} \, e^{-nt^2/2}
E(t) \, \frac{dt}{t}.
\]

\begin{theorem} \label{thm:I12B}
If $s \le \min(1,(q_1n)^{-1/4})$ (so that previous theorem holds) then
\[
\left| I_1 - I_2 \right | \le  \frac{p_0 q_1 }{ 2n}, 
\]
\end{theorem}

\begin{proof}
From everything above (e.g., as in the proof of Theorem~\ref{thm:I01B})
\begin{align*}
\left |  I_1 - I_2 \right |  & \le 
\frac{1}{2\pi  } \, \int_{-s}^{s} \, 
\left|e^{nR} - 1 \right| \, e^{-nt^2/2} \, E(t) \frac{dt}{t} \\
 & \le 2 \cdot \frac{1}{2\pi} \cdot \frac{\pi}{2} \cdot  n p_0 q_1
\, \int_{0}^{s} \, 
t^4 \, e^{-nt^2/2} \, \frac{dt}{t} .
\end{align*}
WGFF2 says that the last integral is equal to 
\[
\frac{1}{n^2} \left (1 - e^{-n s^2/2}\, (1+n s^2/2 ) \right).
\]
An easy calculus exercise shows that the factor in parentheses is between 0 and~$1$, and
the upshot is that the $|I_1-I_2|$ is bounded by $p_0 q_1/(2n)$ as claimed.
\end{proof}

\remark{: }  The last factor could be retained explicitly, giving a 
better estimate, especially when~$n$ is small.  
However, we ignore this because of (a) the general philosophy
of not worrying too much about small constant factors, and (b) if one is struggling with
having to take $n$ large in an unfavorable situation then the factor will be very close to 1 anyway.

\subsection{Eliminating $E$}

Define $I_3$ to be the result of erasing $E(t)$ in the integrand of $I_2$:
\BE{eq:I3}
I_3 = \frac{1}{2\pi i  } \, \int_{-s}^{s} \, e^{iu} \, e^{-nt^2/2} \, \frac{dt}{t} .
\EE

\begin{theorem} \label{thm:I23B}
If $s \le \sigma\pi/3$ then
\[
|I_2-I_3| \le \frac{p_1}{\mu_2 \, n}, \quad \mbox{where} \; p_1 := \frac{3(\pi-3)}{\pi^3} \simeq .0136997\ldots .
\]
\end{theorem}

\begin{proof}
The power series $t/\sin(t)$ is even and has positive coefficients, so we apply the OB
to it and get
\[
\frac{t}{\sin(t)} - 1 \le t^2 \; \left(\frac{s/\sin(s)-1}{s^2}\right), \quad \mbox{if } \; |t| \le s. 
\]
Taking $s = \pi/6$ (for simplicity) gives $t/\sin(t)-1 \le 12(\pi-3)/\pi^2 \, t^2$ if $|t| \le \pi/6$.
Replacing $t$ by $t/(2\sigma)$ and doing a little algebra gives
\[
\left|E(t)-1\right| \le \frac{\pi p_1 t^2}{\mu_2}, \quad \mbox{if } \; |t| \le \pi\sigma/3.
\]
The theorem now follows from WGFF2:
\[
\left| I_2 - I_3 \right | 
\le  \frac{1}{\pi  } \, \int_{0}^{s} \, 
e^{-nt^2/2} \, \frac{\pi p_1 t^2}{\mu_2} \, \frac{dt}{t} 
\le  \frac{p_1}{ \mu_2 } \, \int_{0}^{\infty} \, 
e^{-nt^2/2} \, t^2 \, \frac{dt}{t} 
= \frac{p_1}{\mu_2n}.
\]
\end{proof}

\subsection{Eliminating $\sin$}

The integral of the odd function $\cos(u)/t$ on the symmetric interval $t \in [-s,s]$
is 0 (using the earlier principal value convention).
Then $e^{iu} = \cos(u) + i \sin(u)$ can be replaced $i\sin(u)$ and therefore
\BE{eq:I4}
I_3 = \frac{1}{2\pi  } \, \int_{-s}^{s} \, \sin(u) \, e^{-nt^2/2} \, \frac{dt}{t} .
\EE
For small~$t$, the quantity $u = \beta t - n \nu_3 t^3/6 $ is small, and
the approximation $\sin(u) \simeq u$ motivates defining
\BE{eq:I5}
I_4 := \frac{1}{2\pi  } \, \int_{-s}^{s} \, 
u \, e^{-nt^2/2} \, \frac{dt}{t} .
\EE

\begin{theorem} \label{thm:I34B}
\[
\left |I_3 - I_4\right| \le \frac{q_5}{\sqrt{2\pi} \, n^{3/2}}
\]
where
\[
q_3 = |\beta|, \quad q_4 = |\nu_3|/6, \quad
q_5 = \frac{q_3^3}{6} + \frac{3 q_3^2 q_4 }{2} + \frac{15 q_3 q_4^2}{2}+\frac{35q_4^3}{2}.
\]
\end{theorem}

\begin{proof}
For any real number $x$, $|\sin(x)-x| \le |x|^3/6$.  
Applying this to $u = \beta t - n\nu_3 t^3/6$ gives
\[
|\sin(u)-u| \le \frac{|u|^3}{6} \le \frac{1}{6} 
\left ( q_3^3 t^3 + 3 q_3^2 q_4  n t^5  + 3 q_3q_4^2 n^2 t^6  +q_4^3 n^3 t^9  \right ).
\]
The claimed inequality follows by integrating and using the second WGFF1, i.e.,
\BE{eqn:wgff1}
\int_0^\infty t^{2k-1} \, e^{-ct^2/2} \, \frac{dt}{t} = c^{-k+1} \, (2k-3)!! 
\sqrt{\frac{\pi}{2c}} , \\
\EE
for $k = 2,3,4,5$ to get
\begin{align*}
|I_4 -I_3| & \le \frac{1}{2\pi  } \, \int_{-s}^{s} \, 
\frac{|u|^3}{6} \, e^{-nt^2/2} \, \frac{dt}{t} \\
& \le  \frac{1}{6\pi  } \, \int_{0}^{\infty} \, 
\left ( q_3^3 t^3 + 3 q_3^2 q_4  n t^5  + 3 q_3q_4^2 n^2 t^6  +q_4^3 n^3 t^9  \right )\, e^{-nt^2/2} \, dt \\
& = \frac{1}{6\pi  } \; \sqrt{\frac{\pi}{2n}} \;
\left ( \frac{q_3^3 + 9 q_3^2 q_4   + 45 q_3q_4^2  +105 q_4^3 }{n}  \right ).
\end{align*}
\end{proof}

\subsection{Extending to the whole real line}

The next (and final!) integral is obtained by extending the interval 
of integration to the whole real line, i.e.,
\[
I_5  := \frac{1}{2\pi}  \, \int_{-\infty}^{\infty} \, u \, e^{-nt^2/2} \, \frac{dt}{t} .
\]
This integral can be evaluated immediately using WGFF1 above for $k = 1$ and $k = 2$:
\begin{align*}
I_5 & = \frac{1}{2\pi  } \, \int_{-\infty}^{\infty} \, (\beta - n\nu_3 t^2 /6) \, 
           e^{-n t^2/2} \, dt \\
    & = \frac{\beta}{\pi  } \, \int_{0}^{\infty} \, e^{-nt^2/2} \, dt -
    \frac{n\nu_3}{6\pi  } \, \int_0^{\infty} \, t^2  \, e^{-nt^2/2} \, dt\\
    & = \frac{1}{\sqrt{2\pi n}} \, \left(\beta - \frac{\nu_3}{6}\right) =
\frac{\LL}{\sqrt{2\pi n}}.
\end{align*}

\begin{theorem} \label{thm:I45B}
\[
\left | I_4 - I_5 \right | \le 
\left(\frac{q_3+q_4  + q_4 n s^2}{\pi ns}\right)  \; 
e^{-n s^2/2}.
\]
\end{theorem}

\begin{proof}
The two tails of $I_5-I_4$ are equal.  Multiplying by 2 and using the
last two parts of WGFF3 gives
\begin{align*}
|I_5-I_4| & \le \frac{1}{\pi } \, \int_{s}^{\infty} \, 
(q_3 + q_4 n t^2) \, e^{-nt^2/2} \, dt \\
& \le \frac{q_3}{\pi} \cdot \frac{e^{-ns^2/2}}{ns} + \frac{q_4n}{\pi} \cdot 
\frac{e^{-ns^2/2}}{n^2s} \cdot (1+ns^2) \\
& = \left(q_3+q_4 (1 + n s^2 )\right) \; 
\frac{e^{-n s^2/2}}{\pi ns }.
\end{align*}
\end{proof}

\subsection{Finishing the proof}

During the proof the constant $s$ was required to be smaller than $\pi \sigma, 1, 
\pi \sigma/3,$ and $1/(q_1n)^{1/4}$, which can be summarized as
\BE{sbound}
s \le \min(1,\pi \sigma/3, (q_1n)^{-1/4}).
\EE

We remind the reader of the notation and then state the main theorem: $X$ is
a lattice random variable with mean~0, span~1, shift~$a$, and finitely many
values; $X[n]$ denotes the sum of $n$ IID copies of~$X$.
The standard deviation of $X$ is $\sigma$, and the moments and normalized
moments are denoted $\mu_k = \E X^k$, $\nu_k = \mu_k/\sigma^k$.
Moreover, $\beta = \beta_n$ is shorthand for $(1-\{na\})/\sigma$.
Various further constants are defined as follows:
\[
\begin{array}{l@{\qquad}l}
 p_0 = e-1 \simeq 1.71828 \ldots & {\displaystyle p_1 := 3(\pi-3)/\pi^3 \simeq .0136997\ldots}  \\
 {\displaystyle q_1 := \frac{1}{5} + \frac{\nu_4}{24}} & 
 {\displaystyle q_2 := \frac{p_0  q_1}{2} + \frac{p_1}{\mu_2}}\\
 q_3 := |\beta| & {\displaystyle q_4 = |\nu_3|/6}  \\
 {\displaystyle q_5 := \frac{q_3^3}{6} + \frac{3 q_3^2 q_4 }{2} + 
    \frac{15 q_3 q_4^2}{2}+\frac{35q_4^3}{2}}  & {\displaystyle r  := \frac{16m}{(\pi C \sigma)^2}}  \\
\end{array}
\]
where, as earlier, $C$ is the $L_1$ norm of a certificate $\{c_x\}$ for $X$ (i.e., a
guarantor that~$X$ has span~1), and $m = \min(p_x)$ is
the minimum probability of the values of~$X$ that are in the support of the certificate.

\begin{theorem}
\label{thm:main}
Fix a positive integer~$n$ and let $s$ be a positive real number~$s$ that satisfies~(\ref{sbound}).
Then
\[
\Prb(\Xn < 0)  = \frac{1}{2} - \frac{\LL}{\sqrt{2\pi n}} + \EL
\]
where
\[
\LL = \beta-\frac{\nu_3}{6}
\]
and
\[
|\EL| \le \frac{q_2}{n} + \frac{e^{-nr/2}}{2nr} + \frac{q_5}{\sqrt{2\pi} n^{3/2}} + e^{-ns^2/2} \; 
\left (p_0q_1s^2 + \frac{1}{ns^2} + \frac{2 p_0q_1}{n} + \frac{q_3 + q_4}{\pi ns} + \frac{q_4s}{\pi}\right ).
\]
\end{theorem}

\begin{proof}
The bound (\ref{sbound}) above guarantees that all of the results that required $s$ to be
small enough are valid.
If $I(j,j+1)$ denotes the bound proved above on the difference between $I_j$ and $I_{j+1}$ then
\[
|\EL| \le I(0,1) +I(1,2) +I(2,3) +I(3,4) +I(4,5).
\]
The last four terms were proved in 
Theorem~\ref{thm:I12B},
Theorem~\ref{thm:I23B},
Theorem~\ref{thm:I34B},
and Theorem~\ref{thm:I45B} respectively,
giving
\[
I(1,2) +I(2,3) +I(3,4) +I(4,5) = \frac{q_2}{n} + \frac{q_5}{\sqrt{2\pi} \, n^{3/2}} +
e^{-ns^2/2} \, \left ( \frac{q_3+q_4+q_4 ns^2}{\pi n s} \right ).
\]
We improve the bound for $|I_0-I_1|$ given in Theorem~\ref{thm:I01B} by writing
\[
\int_s^{\pi\sigma} |g(t)|^n \frac{dt}{t} =
\int_s^{1} |g(t)|^n \frac{dt}{t} +
\int_1^{\pi\sigma} |g(t)|^n \frac{dt}{t}
\]
and apply Theorem~\ref{thm:I01B} to the last integral and Corollary~\ref{cor:tail} to the integral
from $s$ to~$1$.  The result is
\[
|I_1-I_0| \le  e^{-ns^2/2} \, \left ( \frac{1}{ns^2} + \frac{2p_0q_1}{n}+p_0q_1s^2 \right) 
+ \frac{e^{-nr/2}}{2nr}.
\]
The theorem follows with a little algebra.
\end{proof}

In the next section this will be used to state a theorem about the tilt, and some simple assumptions
will be made that will simplify and clarify this expression.

One possible major improvement to the estimate would come from
using the next higher order Edgeworth expansion. 
The error term would become $O(1/n^2)$.
However, our hunch is that this would introduce two further problems: the number of
terms in the algebraic expressions
would explode, and the upper bound on $s$ would probably be $O(n^{-1/6})$, so it 
is not immediately clear that the resulting $n_0$ would be much better.

\section{Tilt}

The goal of this section is to apply the Main Theorem to the tilt.

The first subsection makes a few observations about the error bounds.  The second reverts to the
case of dice, giving a theorem about the tilt.  Finally, the third subsection considers some
numerical examples.

\subsection{Observations}

With the notation in Theorem~\ref{thm:main}, fix~$n$ and consider the error bound
as a function of~$s$.  The only two terms that are not obviously decreasing as $s$
increases are (constant multiples of) $se^{-ns^2/2}$ and $s^2e^{-ns^2}/2$.
Differentiating with respect to~$s$ shows that these are both decreasing if $ns^2\ge 2$.
Since this is a very mild assumption, we will just assume it from now on.
This implies that the optimal $s$ is just the smallest value in (\ref{sbound}) above.

In that bound
\[
s \le \min(1,\pi \sigma/3, (q_1n)^{-1/4}).
\]
the last quantity will usually be the smallest.
To simplify and focus the notation, assume that this is the case, i.e., that 
\[
\frac{1}{(q_1n)^{1/4}} \le \min(1,\pi \sigma/3).
\]
Assuming this, requiring $ns^2\ge 2$ as above, setting $s = (q_1n)^{-1/4}$, and adopting
the notational shorthand
\[
\eta = 2\sqrt{n/q_1}
\]
(so that $ns^2/2 = \eta $), allows a restatement of the error bound as follows: 

\noindent If 
\[
n \ge \max\left(\frac{q_1}{4},\frac{1}{q_1},\frac{81}{q_1\pi^4\mu_2^2}\right)
\]
then
\[
|\EL| \le \frac{q_2}{n} + \frac{e^{-nr/2}}{2nr} + \frac{q_5}{\sqrt{2\pi} n^{3/2}} + e^{-\eta } \; 
\left (\frac{p_0+1}{2\eta } + \frac{2 p_0q_1}{n} + \frac{1}{\pi \sqrt[4]{q_1n}}\left(\frac{q_3 + q_4}{2\eta } + q_4\right )\right).
\]

For the sake of the next section we note that if $X$ is replaced by $-X$ in the theorem
then, as should be expected, the bounds on the error are unchanged.
Indeed, the only relevant change is that the sign of $a$ is negated, so that $q_3$ might
change.
However, if $x = \{na\}$ then and $y = \{-na\}$ then either $x = y = 0$, or $y = 1-x$.  
If $x = y = 0$ then $q_3$ is obviously unchanged.  If $y = 1-x$ then
\[
q_3 = \frac{|1/2-x|}{\sigma} = 
\frac{|-1/2+y|}{\sigma} = 
\frac{|1/2-y|}{\sigma} 
\]
and $q_3$ is again unchanged.

\subsection{Tilt for dice}

For the sake of applications it is convenient to return to looking at the tilt in the case of dice.
Let $X$ be a bounded integer-valued random variable
with span~$b$, shift~$a$, mean~$0$, and moments $\mu_k = \E X^k$,
$\sigma^2 = \mu_2$, and $\nu_k = \mu_k/\sigma^k$.

Apply the Main Theorem (using the reformulation in the preceding section)
to $X/b$.  Scale-invariant quantities are unchanged, but
$\sigma(X/b)  = \sigma/b$.
The quantity $\beta$ is unchanged, but note that
\[
\beta = \frac{1/2 - \{na/b\}}{\sigma/b} = \frac{b/2 - b\{na/b\}}{\sigma} 
= \frac{b/2 - na \bmod b}{\sigma}.
\]
The other constants are changed only in so far as $\sigma$ has to be replaced by $\sigma/b$:
\[
\begin{array}{l@{\qquad}l}
 p_0 = e-1 \simeq 1.71828 \ldots & {\displaystyle p_1 := 3(\pi-3)/\pi^3 \simeq .0136997\ldots}  \\
 {\displaystyle q_1 := \frac{1}{5} + \frac{\nu_4}{24}} & 
 {\displaystyle q_2 := \frac{p_0  q_1}{2} + \frac{b^2 p_1}{\mu_2}}\\
 q_3 := |\beta| & {\displaystyle q_4 = |\nu_3|/6}  \\
 {\displaystyle q_5 := \frac{q_3^3}{6} + \frac{3 q_3^2 q_4 }{2} + 
    \frac{15 q_3 q_4^2}{2}+\frac{35q_4^3}{2}}  & {\displaystyle r  := \frac{16b^2m}{(\pi C \sigma)^2}}  \\
\end{array}
\]
where, as earlier, $C$ is the $L_1$ norm of a certificate $\{c_x\}$ for $X$, i.e., a
collection $\{c_x\}$ for a set of values~$x$ of~$X$ such that $\sum c_x = 0, \sum c_x x = b$.
Note that the inclusion of $b$ in the formulae does not actually change the 
values of the constants.  For instance, in the span-1 case of $X/b$ the standard
deviation is $\sigma/b$, where now ``$\sigma$'' refers to the span of~$X$ and not that
of $X/b$.

\begin{theorem}\label{Tnbound}
Let $X$ be as above.
Fix~$n$ with $n \ge \max\left(q_1/4,1/q_1, 81b^4/(q_1 \pi^4 \mu_2^2)\right)$.
Then
\[
T_n = \frac{1}{\sqrt{2\pi n}} \left ( \frac{(-na)\bmod b - (na \bmod b)}{\sigma} - \frac{\nu_3}{3} \right ) + E
\]
where
\[
|E| \le \frac{2q_2}{n} + \frac{e^{-nr/2}}{nr} + \frac{2q_5}{\sqrt{2\pi} n^{3/2}} + e^{-\eta} \; 
\left (\frac{1+p_0}{\eta} + \frac{4 p_0q_1}{n} + \frac{1}{\pi \sqrt[4]{q_1n}}
\left(\frac{q_3 + q_4}{\eta} + 2q_4\right )\right)
\]
and $\eta = 2\sqrt{n/q_1}$ so that $ns^2/2 = \eta$.
\end{theorem}

\begin{proof}
The theorem follows from the fact that
\[
T_n = \Prb(\Xn> 0) - \Prb(\Xn < 0) = \Prb((-X)[n]<0)-\Prb(\Xn<0)
\]
and the earlier remarks on the error bounds.
\end{proof}

\subsection{Examples}

We apply the above results to several examples, comparing the actual $n_0$ at which 
asymptopia arrives to various approximations that emerge from Theorem~\ref{Tnbound}.

Let $X$ be the random variable that takes values $-3$, $1$, $5$ with respective probabilities
$1/2$, $1/4$, $1/4$, so that its PGF is
\[
X(z) = (2z^{-3}+z+z^5)/4
\]
(it is convenient to identify dice with their PGFs).
Then $X$ has mean 0, span 4 and shift 1, so there are really four cases that 
have to be considered: $n$ going to infinity through integers that are $c \bmod 4$
for $c = 0,1,2,3$;  in the table below data connected with case $c$ is on the line
labeled $X_c$.

With~$c$ fixed, let $n_0$ be the smallest integer such that
the sign of $T_n$ is equal to the sign of~$L$ for all $n \ge n_0$, $n \equiv c \bmod b$, 
i.e., $n_0$ is the exact point at which asymptopia has arrived in the congruence
class~$c$.

The term $2q_2/n$ in the error bound for $T_n - L/\sqrt{2\pi n}$ is
unavoidable in any bound obtained by using Edgeworth expansions as above, and we
will call this the ``principal term'' of the error, motivated by the fact that in
sufficiently favorable circumstances it will be the dominant term.
In particular, it is impossible for us to {\em prove} that asymptopia has arrived unless
$n$ is large enough so that
\[
\sqrt{2\pi n} \; \frac{ 2 q_2}{n} < |L|, \quad \mbox{i.e.,} \quad
n \ge n_1 :=  \frac{8\pi q_2^2}{L^2}.
\]
Thus $n_1$ is a lower bound on asymptopia arrival
that could be proved by our techniques.
Given the nature of some of the error bounds, it
is reasonable to expect that $n_1$ might might actually be larger than $n_0$ in many cases; 
however, a later example gives an instance where $n_1$ is lower than~$n_0$.

Finally, let $n_2$ be the number, produced by using the error estimates for a fixed congruence
class modulo~$b$ directly as stated in
Theorem~\ref{Tnbound}, such that for $n \ge n_2$ (and in the given congruence class), 
$\sqrt{2\pi n} \; \EB(n) < |L|$, where $EB(n)$ is the error bound in the theorem.
In other words, Theorem~\ref{Tnbound} can be used to show that asymptopia has arrived by $n_2$
in the given congruence class.

The values found  using (a computer and) Theorem~\ref{Tnbound} are:
\[
\begin{array}{rrrrr}
    &  L        &   n_0   &   n_1 &       n_2   \\ \hline
X_0 & -0.16446  &    4    &   59  &        74   \\  
X_1 &  0.43856  &    5    &    9  &        37   \\
X_2 & -0.16446  &    2    &   59  &        70   \\
X_3 & -0.76748  &    3    &    3  &        27   \\
\end{array}
\]
As expected, smaller values of $|L|$ require larger $n$.
A close examination of the seven terms in the error bound show that the last five
rapidly become negligible as $n$ becomes large, whereas the first two terms --- the principal term 
and the tail bound, $e^{-rn/2}/(rn)$, are significant.
The tail bound can be decreased, as we will see below, by looking at the tail integral more
closely.
However, in the $X_c$ cases the amount of computer time required to compute the tilts up
to~$n_2$ is negligible, so we will not bother trying to improve the tail bound in
these cases, despite the gap between $n_0$ and $n_2$.

Let $Y$ be the random variable that takes the value $-8$ with probability $1/2$, the
value $0$ with probability $1/18$, and the value 9 with probability $4/9$; its PGF is
\[
Y(z) = (9z^{-8}+1+8z^9)/18.
\]
The span of $Y$ is 1.  One (of the several) indications that this might be a problematic 
case is that the probability of $0$ is small, and the span changes to 17 if this probability is
set to 0 (while suitably rebalancing the other two probabilities).

To give a sense of the various components of the error, write
\[
\EB = \frac{2q_2}{n} + \frac{e^{-rn/2}}{rn}+ \TR
\]
where $\EB$ is the total error bound and $\TR$ (``the rest'') is the sum of five other terms.
In the following table $n$ is either close to $n_0 = 761$, $n_1 =682$, or $n_2 = 182024$.
The columns are: $n$, the total error $\EB$, the principal error, and the tail error.
(All errors have been multiplied by $\sqrt{2\pi n}$ to make comparison with $L$ easy). In all cases, 
$\TR$ is less than $10^{-5}$ and it is not tabulated.
All decimal expansions are \emph{truncated} rather than rounded, and the constant 
$L$ is equal to $-0.0404226\ldots$.
\[
\begin{array}{cccc}
n      & \EB       & 2q_2/n    & e^{-rn/2}/(rn) \\
681    & 1128.163 & 0.0404310 & 1128.122 \\
682    & 1127.289 & 0.0404013 & 1127.248 \\
761    & 1063.690 & 0.0382468 & 1063.652 \\
182023 & 0.040423 & 0.0024730  & 0.0379500 \\
182024 & 0.040421 & 0.0024729  & 0.0379483 \\
\end{array}
\]

The delicate nature of this case is illustrated by the graph of
$\sqrt{2\pi n}\, T_n$ for $n$ from  1 to 800, and for $n$ from 740 to 800, given
in Figure~1 and Figure~2.

\begin{figure}
\includegraphics[width=\textwidth,height=2in]{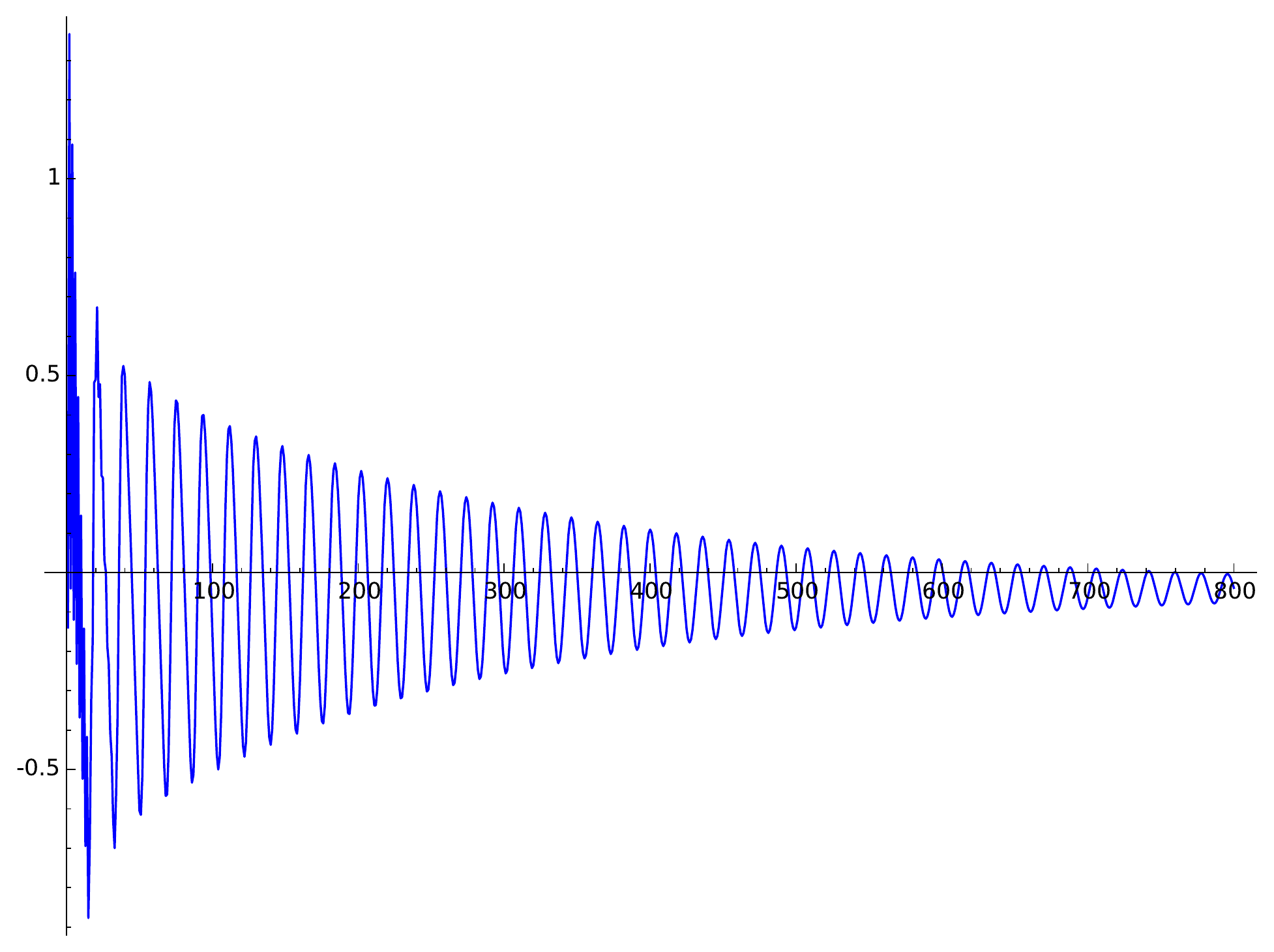}
\caption{$\sqrt{2\pi n}\, T_n$, $1 \le n \le 800$}
\end{figure}

\begin{figure}
\includegraphics[width=\textwidth,height=2in]{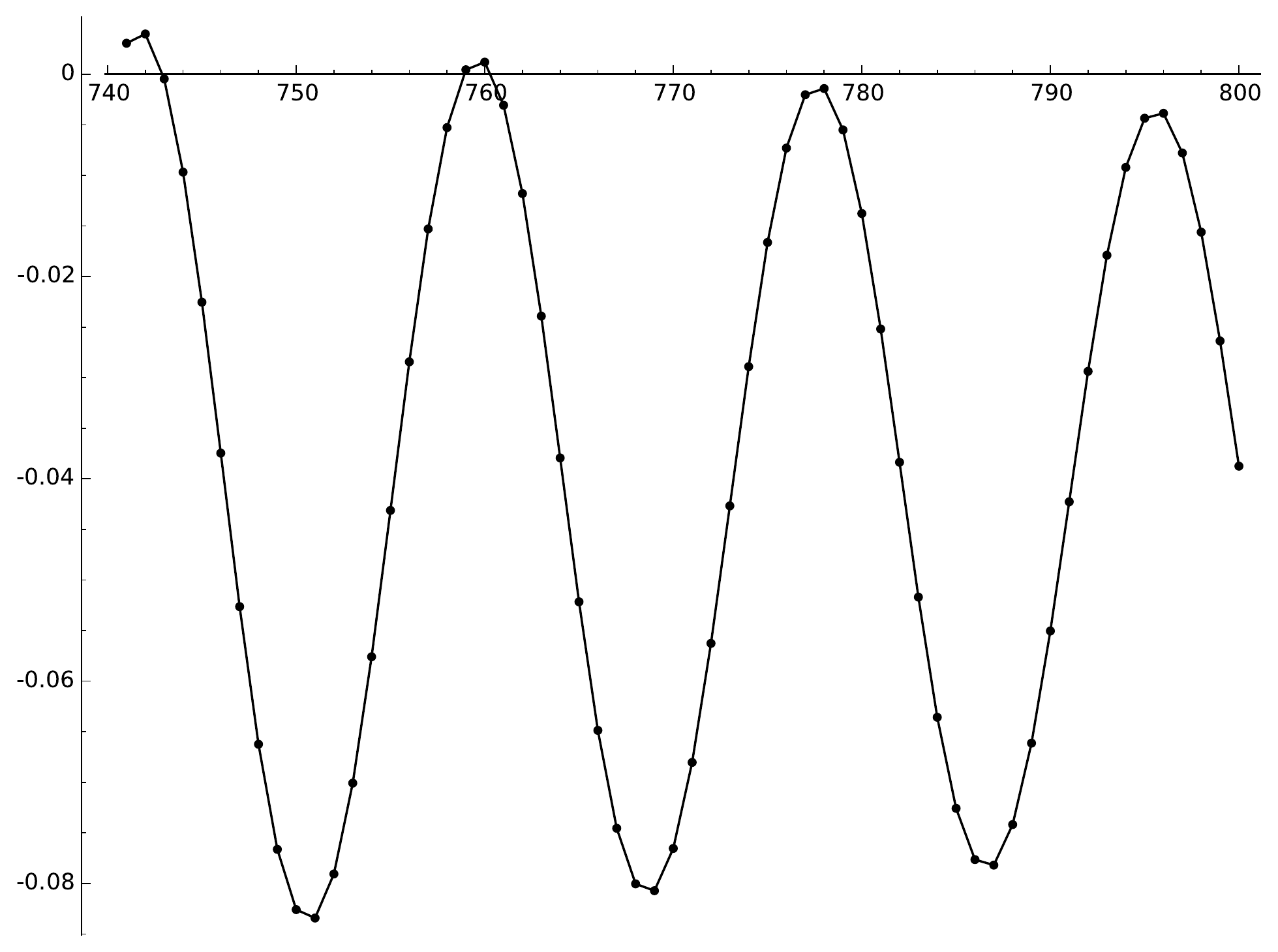}
\caption{$\sqrt{2\pi n}\, T_n$, $740 \le n \le 800$}
\end{figure}

Clearly $Y$ ``nearly''  has span~17, so that the graph is the result of applying
a damping function to a function that is periodic of period~17.
However, at $n_0 = 761$ the damping finally forces to the tilt to become, and forever stay,
negative.
The actual numerical values in the vicinity of~$n_0$, and the next local maxima, are:
\[
\begin{array}{rrrrr}
    n  &     759  & 760 & 761 & 762 \\
 \sqrt{2\pi n}\,T_n & 0.000439 & 0.001195 & -0.003066 & -0.011796 \\
  & & \cdots & & \\
    n  & 776 & 777 & 778 & 779 \\ 
 \sqrt{2\pi n}\,T_n & -0.007300 & -0.002028 & -0.001415 & -0.005505 \\
\end{array}
\]

The size of the tail bound near $n_0$ was a surprise to us; it is 
15 times the size of the principal error term.
Moreover, we had thought that
Corollary~\ref{cor:tail} shifted the extreme tail bound to an exponential term
of the form $e^{-cn}$ rather than $e^{-c\sqrt{n}}$, and that this would be good enough.
The key point is of course that the constant $r$ is uncomfortably small.
We wondered whether the alternate $r$ in \cite{Benedicks} would be better, but in the
case of $Y$ Benedick's constant is very slightly worse than ours, and gives 
essentially the same $n_2$.
We tried to replace $r$ with the optimal value 
\[
r_{\rm opt} = \min_{t \in [-\pi,\pi]} \frac{1-|f(t)|}{t^2}.
\]
This gave a factor of improvement to $n_2$ of somewhat more than~2:

\[
\begin{array}{lcr}
 & r & n_2 \\
\mbox{Benedicks}   & .0026144\ldots & 194{,}081 \\
\mbox{ours}  & .0028144 \ldots & 182{,}024\\
\mbox{optimal} & .0055834\ldots & 88{,}181 \\
\end{array}
\]

The graph of the absolute value $|f(t)|$ of the CF shows what the problem is.
The goal is to find an upper bound for the integral of $|f(t)|^n/t$ outside
$[-1/\sigma,1/\sigma]$, as $n$ gets large.
The function $1-rt^2 \le e^{-rt^2}$ is a natural choice, but does not
work well for a characteristic function whose secondary peaks are so high.

\begin{figure}[b]
\includegraphics[width=\textwidth,height=2in]{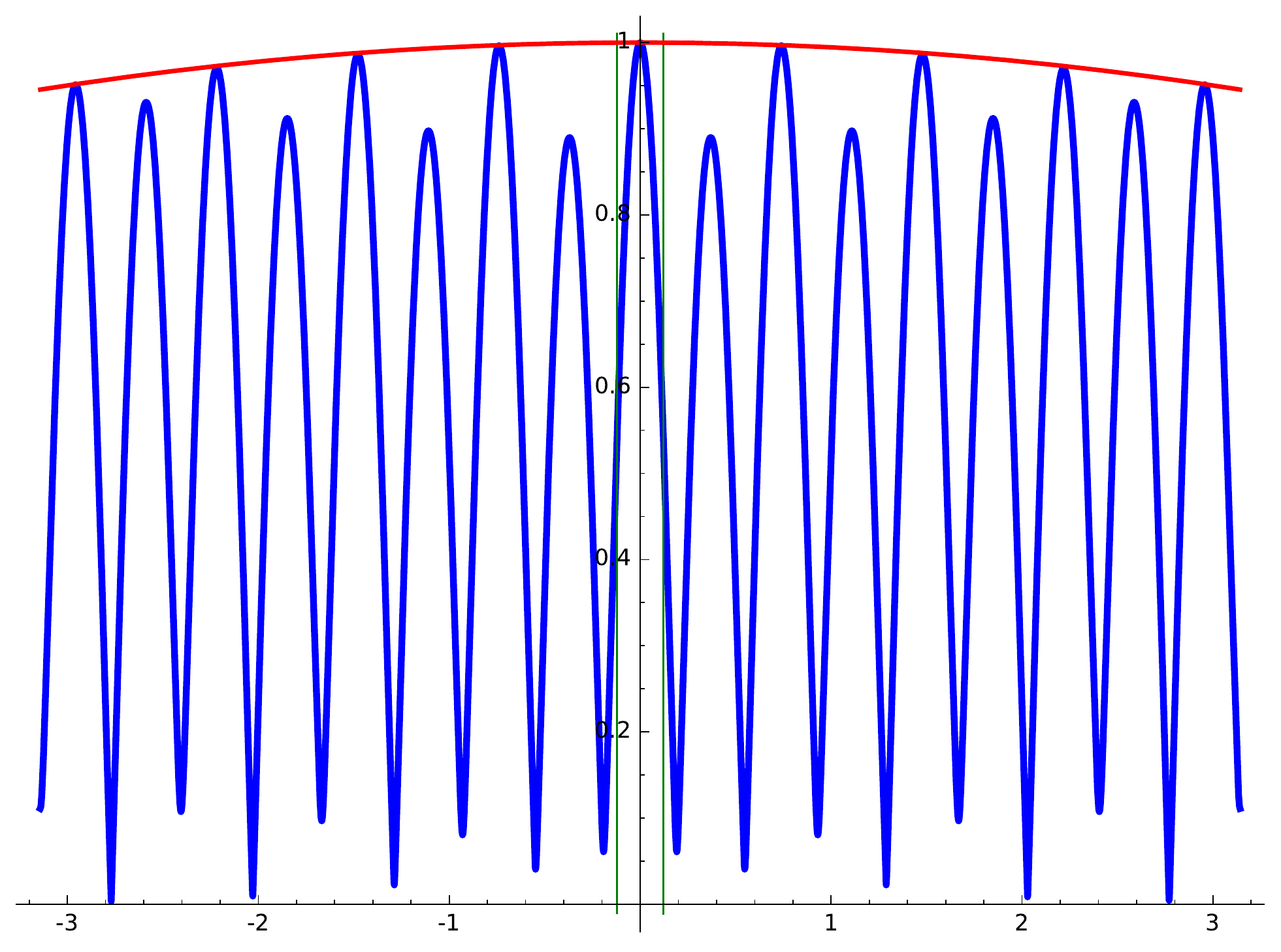}
\caption{$|f_Y(t)|$}
\end{figure}

There are several things that can be done to improve the bound on the
tail integral.  
The following extremely simple device made a dramatic improvement compared to the
$n_2$ given above.
The heights of the peaks in the graph of $|f(t)|$ (starting from $0$ and 
moving to the right) are
\[
1, 0.88989, 0.99645, 0.89768, 0.98621, 0.91204, 0.97048, 0.93077, 0.95118.
\]
The third peak is dominant. It occurs at $t = t_0 = 4\pi/17$ and has 
height $h_0 = .99645$.  The fifth peak has height $h_1 = .98621$ and is 
the second most dominant peak.  Let $f^*$ be
the function that is the constant $h_1$ on $[-\pi,-1/\sigma] \cup [1/\sigma,\pi]$ except 
for the small
section of the parabola $h_0\cdot(1-34(t-t_0)^2)$ that is above the line $y= h_1$,
and the mirror image of this parabola section at $t = -t_0$.
One can check that $f^*$ is an upper bound on $|f|$ outside of $[-1/\sigma,1/\sigma]$, 
and it is easy to estimate the integral of $f^*(t)^n/|t|$ on $1/\sigma \le |t| \le \pi$.

Using this tail bound we get an improved $n_2' = 1455$, which was much smaller
than we had expected.
The only significant contributions to the error are the $2q_2/n$ term
and the tail bound, which are, respectively, $0.02766$, and $0.012709$
so that the total error is just less than $|L|$:
\[
0.0276602867+0.01270946 = .040370408 < |L| = .0404226.
\]
All other contributions to the error are less than $10^{-6}$, and the tail
bound has been reduced to under half of the primary contribution.

One can push this further by considering all of the peaks and lowering the
``threshold'' to, say, the value of the CF at $1/\sigma$,
i.e., the value of the normalized CF $g(t)$ at $t = 1$.  This seems to
move $n_2'$ further down to an $n_2''$ just above 1200 (and we think that this is about the
limit of what can be done).
However, the time required to program this correctly far exceeds the time
that a computer takes to compute the tilts between 1200 and 1455, so 
our general philosophy says that it would be silly to implement this improvement.

Finally, we consider an example that originally motivated this investigation.
Consider three nontransitive dice,
that appeared in one of Martin Gardner's columns many  years ago,
whose PGFs are
\[
A = (z^2+z^6+z^7)/3, \qquad  B = (z^1+z^5+z^9)/3, \qquad C = (z^3+z^4+z^8)/3.
\]
Write $A>B$ to mean that
$\Prb(A>B) > \Prb(B>A)$, i.e., it is more likely that a roll of $A$ is larger 
than a roll of $B$ than the reverse.  This is the same thing as saying that
the tilt $T(A-B)$ of the difference $U = A-B$ is positive.  Note that the PGF
of $U$ is $A(z)B(z^{-1})$.  Similarly, let $V = B-C$ and $W = C-A$. 
It turns out that
\[
A > B, \; B > C, \; C > A, \quad \mbox{whereas} \quad A[2] < B[2], \; B[2] < C[2], \; C[2] < A[2].
\]
In fact, the various dominance orders oscillate until $n = n_0 = 8$ when $A$ becomes
dominant, and $B$ dominates~$C$, in the sense that
\[
A[n] > B[n], \; A[n]>C[n], \; B[n] > C[n], \quad \mbox{for all} \quad n \ge n_0.
\]
This can be verified by computer for $n$ as large as your hardware can go,
but to {\em prove} that asymptopia arrives at $n_0 = 8$ we use Theorem~\ref{Tnbound}.
It turns out that 
\[
U(z) = A(z)B(z^{-1}) = B(z)C(z^{-1}) = V(z)
\]
so there are really only two dice to which that theorem has to be applied: $U$ and $W$.
By now we can give a guess as to how large $n$ will have to be, namely, we
find the needed error $|L|$ and then choose $n$ large enough so that
\[
\sqrt{2\pi n} \; T_n \simeq
\sqrt{2\pi n} \; \frac{2 q_2}{n} < |L|, \; \mbox{i.e.}, \quad n > \frac{8\pi q_2^2}{L^2}.
\]
Since $q_2 = 1/5 + \nu_4/24$, for many random variables it is 
reasonable to bound $q_2$ by $1/4$, so that $n$ has to be at least as 
large as $\pi/(2L^2)$.

For $W$ we find that $L = 0.033310\ldots$, which is unusually small.
The approximation $n \simeq 8\pi q_2^2/L^2$ suggests $n_0 \le 1407$,
and in fact in this case $L$ is so small that $n$ is large enough so
that all other terms of the error are negligible.
In other words, the $n$ required by the principal error term so large that this 
gives the best possible value.

For $U$, the limit $L$ is larger, namely $L \simeq -0.14028\ldots$.
In this case the bound implied by the principal term is $n = 83$.
Although the main tail bound term is then large, we can apply the
earlier techniques of piecewise bounding the characteristic function,
to show that $n_0 \le 83$.  In other words, this shows that the smallest
possible bound is achievable with a bit more work.

All of the above experiments are summarized in the following table.
\[
\begin{array}{rrrrrrr}
    &  L        &   n_0   &   n_1 &       n_2   & n_2' & n_2'' \\ \hline
X_0 & -0.16446  &    4    &   59  &        74   &      &       \\  
X_1 &  0.43856  &    5    &    9  &        37   &      &       \\
X_2 & -0.16446  &    2    &   59  &        70   &      &       \\
X_3 & -0.76748  &    3    &    3  &        27   &      &       \\
Y   & -0.040422 &   761   &  682  &  182{,}024  & 1455 & 1206  \\
U   & -0.14028  &    9    &   83  &     1933    &   83 &       \\
W   &  0.03333  &    5    & 1407  &     4591    & 1407 &       \\
\end{array}
\]
where $n_2'$ is the result of replacing the tail bound using a little bit of work
(e.g., roughly where one could imagine that the necessary estimates could be
verified by hand, as in the simple improvement for~$Y$ above), and 
$n_2''$ is the result of replacing the tail bounds by
bounds that require a computer to perform all of the verifications, as in  the
more complicated improved estimate for $Y$ earlier.

We close with some further comments.

\begin{enumerate}
\item
The primacy of the $1/n$ term is a surprise.
If we take the ``poor man's approximation'' $q_2 \simeq 1/4$,
which is a good approximation unless $X$ takes very large values
with very small probability, then the error is about $1/2n$, independent
of~$X$!  And this says that the lower bound $n_1$, and often the arrival bound~$n_2$,
 is almost entirely determined by the target error $|L|$.
As a first guess, $n_2 \simeq \pi/(2L^2)$, especially if this number is large
enough so that the exponential terms are small; if this approximation to $n_2$ 
isn't that large then
further work might be required to decrease the tail bound.
\item
A curious philosophical difficulty is hiding in the weeds.
The value of $n_0$ becomes ``obvious'' from calculations when the sign
of the tilt becomes constant and stays there for as large an $n$ as one cares
to compute.  However, this gives no hint of how one might prove that this will
continue to be the case, and the point of the work here is to be able to actually
prove an upper bound on $n_0$.
What computer results are admissible in such a proof?
The computation of tilts would seem innocuous to many since any floating point error
can be easily bounded, and the programs are short and ``easily'' proved to be correct
--- the computer is ``just'' doing stable, well-understood arithmetic .
The error bounds in calculating $n_2$ straight from Theorem~\ref{Tnbound} can be 
done by hand, and perhaps the estimates for $n_2'$ could also be done by hand, though few, if any
people would do them nowadays without using a computer; the calculations needed to support
the determination of $n_2''$ seem to be intrinsically even more demanding.
\item
It would be interesting to apply these ideas here to more general situations
(more general RVs, approximation not at the mean, etc.).
Extending to higher order Edgeworth approximations seems viable, but
would require algebraic stamina.  We have thoughts on how this might be
automated. However, as noted earlier, we expect that the lower
bounds on $n$ will have to increase, so that the utility of this approach
isn't entirely clear.
\end{enumerate}

\bigskip

\noindent{\bf Acknowledgments. }  We would like to thank Richard Arratia, Steve Butler,
Persi Diaconis, Larry Goldstein, Fred Kochman, and Sandy Zabell for interesting and
useful feedback.

\bibliography{GP}
\bibliographystyle{amsalpha}

\end{document}